\documentclass{amsart}
\usepackage{setspace, amsmath, amsthm, amssymb, amsfonts, amscd, epic, graphicx, ulem, dsfont}
\usepackage[T1]{fontenc}
\usepackage{multirow}
\usepackage{bbm}
\usepackage{enumerate}

\makeatletter \@namedef{subjclassname@2010}{
  \textup{2020} Mathematics Subject Classification}
\makeatother

\newtheorem{thm}{Theorem}[section]

\theoremstyle{remark}
\newtheorem*{rema}{Remark}

\theoremstyle{definition}

\newtheorem{exa}[thm]{\textbf{Example}}

\newcommand{\R}{\mathbb{R}}

\begin{document}

\title[Fuglede-Putnam theorem]{Yet another generalization of the Fuglede-Putnam theorem to unbounded operators}
\author[M. H. MORTAD]{Mohammed Hichem Mortad}

\thanks{}
\date{}


\address{Department of
Mathematics, University of Oran 1, Ahmed Ben Bella, B.P. 1524, El
Menouar, Oran 31000, Algeria.\newline {\bf Mailing address}:
\newline Pr Mohammed Hichem Mortad \newline BP 7085 Seddikia Oran
\newline 31013 \newline Algeria}

\email{mhmortad@gmail.com, mortad.hichem@univ-oran1.dz.}

\begin{abstract}In this note, we give the most natural (perhaps the simplest ever)
generalization of the Fuglede-Putnam theorem where all operators
involved are unbounded.
\end{abstract}

\maketitle

\section{Introduction}

The original version of the Fuglede-Putnam theorem reads:

\begin{thm}\label{Fug-Put UNBD A BD} If $A\in B(H)$ and if $M$
and $N$ are normal (non necessarily bounded) operators, then
\[AN\subset MA\Longrightarrow AN^*\subset M^*A.\]
\end{thm}

Fuglede \cite{FUG} established Theorem \ref{Fug-Put UNBD A BD} in
the case $N=M$. Then Putnam \cite{PUT} proved the theorem as it
stands. Many would agree that the most elegant proof is the one due
to Rosenblum in \cite{ROS.FP}.

There have been many generalizations of the Fuglede-Putnam theorem
since Fuglede's paper. However, most generalizations were devoted to
relaxing the normality assumption. The generalizations to closed
unbounded $A$ in the theorem above are few. The first result in this
sense is:

\begin{thm}\label{Fug-Put-MORTAD-PAMS-2003} If $A$ is a closed operator and if $N$ is an unbounded normal operator, then
\[AN\subset N^*A\Longrightarrow AN^*\subset NA\]
whenever $D(N)\subset D(A)$.
\end{thm}

In fact, the previous result was established in \cite{MHM1} under
the assumption of the self-adjointness of $A$ (see
\cite{Paliogiannis-NEW-proofs-MORTAD} for a different proof).
However, by scrutinizing its proof, it is seen that only the
closedness of $A$ was needed (the self-adjointness was added to be
used in some subsequent results). Later in
\cite{Mortad-Fuglede-Putnham-All-unbd} the following generalization
was obtained:

\begin{thm}\label{Main Theorem A closed  N M unbounded}
Let $A$ be a closed operator with domain $D(A)$. Let $M$ and $N$ be
two unbounded normal operators with domains $D(N)$ and $D(M)$
respectively. If $D(N)\subset D(AN)\subset D(A)$, then
\[AN\subset MA\Longrightarrow AN^*\subset M^*A.\]
\end{thm}

A similar result was obtained in
\cite{paliogiannis-Fug-Putnam-ALL-UNBD}, under the assumptions
$D(N)\subset D(A)$ and $D(M)\subset D(A^*)$.

The purpose of this note is to give the simplest possible (and most
probably minimal in terms of hypotheses) generalization of this
powerful tool in Operator Theory. Apparently, this is the best
possible generalization as Example
\ref{hghgdfsdzretrutioyoyoypypypy} tells us.

In the end, readers of this paper should have knowledge of linear
unbounded operators, as well as matrices of unbounded operators.
Some useful references are \cite{SCHMUDG-book-2012} and
\cite{tretetr-book-BLOCK} respectively.

\section{A counterexample}

In \cite{Mortad-Fuglede-Putnham-All-unbd}, we provided an explicit
pair of a boundedly invertible and positive self-adjoint unbounded
operator $A$ and a normal unbounded operator $N$ (both defined in
$L^2(\R)$) such that
\[AN^*=NA\text{ but }AN\not\subset N^*A \text{ and }N^*A\not\subset AN\]
(in fact $ANf\neq N^*Af$ for all $f\neq 0$). In other words, $NA$ is
self-adjoint whilst $N^*A$ is not.

Recall that $A$ and $N$ were then defined by
\[Af(x)=(1+|x|)f(x)\text{ and } Nf(x)=-i(1+|x|)f'(x)\]
(with $i^2=-1$) respectively on the domains
\[D(A)=\{f\in L^2(\R): (1+|x|)f\in L^2(\R)\}\]
and
\[D(N)=\{f\in L^2(\R): (1+|x|)f'\in L^2(\R)\}\]
(where the derivative is a distributional one).

This example can further be beefed up in the following sense:

\begin{exa}\label{hghgdfsdzretrutioyoyoypypypy} There is a closed $T$ and a normal $M$ such that $TM\subset MT$ but $TM^*\not\subset M^*T$ and $M^*T\not\subset
TM^*$. Consider
\[M=\left(
      \begin{array}{cc}
        N^* & 0 \\
        0 & N \\
      \end{array}
    \right)\text{ and } T=\left(
                            \begin{array}{cc}
                              0 & 0 \\
                              A & 0 \\
                            \end{array}
                          \right)
\]
where $N$ is normal with domain $D(N)$ and $A$ is closed with domain
$D(A)$ and such that $AN^*=NA$\text{ but }$AN\not\subset N^*A$
\text{ and }$N^*A\not\subset AN$ (as defined above). Clearly, $M$ is
normal and $T$ is closed. Observe that $D(M)=D(N^*)\oplus D(N)$ and
$D(T)=D(A)\oplus L^2(\R)$. Now,
\[TM=\left(
                            \begin{array}{cc}
                              0 & 0 \\
                              A & 0 \\
                            \end{array}
                          \right)\left(
      \begin{array}{cc}
        N^* & 0 \\
        0 & N \\
      \end{array}
    \right)=\left(
              \begin{array}{cc}
                0_{D(N^*)} & 0_{D(N)} \\
                AN^* & 0 \\
              \end{array}
            \right)=\left(
              \begin{array}{cc}
                0 & 0_{D(N)} \\
                AN^* & 0 \\
              \end{array}
            \right)
    \]
where e.g. $0_{D(N)}$ is the zero operator restricted to $D(N)$.
Likewise
\[MT=\left(
      \begin{array}{cc}
        N^* & 0 \\
        0 & N \\
      \end{array}
    \right)\left(
                            \begin{array}{cc}
                              0 & 0 \\
                              A & 0 \\
                            \end{array}
                          \right)=\left(
                            \begin{array}{cc}
                              0 & 0 \\
                              NA & 0 \\
                            \end{array}
                          \right).\]
Since $D(TM)=D(AN^*)\oplus D(N)\subset D(NA)\oplus L^2(\R)=D(MT)$,
it ensues that $TM\subset MT$. Now, it is seen that
\[TM^*=\left(
                            \begin{array}{cc}
                              0 & 0 \\
                              A & 0 \\
                            \end{array}
                          \right)\left(
      \begin{array}{cc}
        N & 0 \\
        0 & N^* \\
      \end{array}
    \right)=\left(
              \begin{array}{cc}
                0 & 0_{D(N^*)} \\
                AN & 0 \\
              \end{array}
            \right)\]
and
\[M^*T=\left(
      \begin{array}{cc}
        N & 0 \\
        0 & N^* \\
      \end{array}
    \right)\left(
                            \begin{array}{cc}
                              0 & 0 \\
                              A & 0 \\
                            \end{array}
                          \right)=\left(
                            \begin{array}{cc}
                              0 & 0 \\
                              N^*A & 0 \\
                            \end{array}
                          \right).\]

Since $ANf\neq N^*Af$ for any $f\neq 0$, we infer that
$TM^*\not\subset M^*T$ and $M^*T\not\subset TM^*$.
\end{exa}

\section{New versions of the Fuglede-(Putnam) theorem}

\begin{thm}\label{Fug-Put-MORTAD-PAMS-2003-2020} If $T$ is a closed operator with domain $D(T)\subset H$, and if $N$ is a normal operator, then
\[TN\subset NT\Longrightarrow TN^*\subset N^*T\]
whenever $D(N)\subset D(T)$.
\end{thm}

\begin{proof}Let $T$ be a closed operator and let $N$ be a normal
operator with $TN\subset NT$ and $D(N)\subset D(T)$. Set
\[A=\left(
                            \begin{array}{cc}
                              0 & 0 \\
                              T & 0 \\
                            \end{array}
                          \right)\text{ and }\widetilde{N}=\left(
      \begin{array}{cc}
        N & 0 \\
        0 & N^* \\
      \end{array}
    \right)\]
    and so $D(A)=D(T)\oplus H$ and $D(\widetilde{N})=D(N)\oplus D(N^*)$.

Clearly
\[A\widetilde{N}=\left(
                            \begin{array}{cc}
                              0 & 0_{D(N^*)} \\
                              TN & 0 \\
                            \end{array}
                          \right)\subset \left(
                            \begin{array}{cc}
                              0 & 0 \\
                              NT & 0 \\
                            \end{array}
                          \right)=\widetilde{N}^*A.\]

Since $\widetilde{N}$ is normal and $D(\widetilde{N})\subset D(A)$,
Theorem \ref{Fug-Put-MORTAD-PAMS-2003} applies and implies that
$A\widetilde{N}^*\subset \widetilde{N}A$, that is
\[\left(
                            \begin{array}{cc}
                              0 & 0_{D(N)} \\
                              TN^* & 0 \\
                            \end{array}
                          \right)\subset \left(
                            \begin{array}{cc}
                              0 & 0 \\
                              N^*T & 0 \\
                            \end{array}
                          \right).\]

In particular, for any $(f,0)\in D(A{\widetilde{N}}^*)$:
$A\widetilde{N}^*f=\widetilde{N}Af$ (with some abuse of notation).
Still in particular, we obtain
\[TN^*f=N^*Tf\text{ for any }f\in D(TN^*)\subset
D(N^*T).\] In other words, $TN^*\subset N^*T$, as needed.
\end{proof}

\begin{rema}The condition $D(N)\subset D(T)$ may not just be dropped
as Example \ref{hghgdfsdzretrutioyoyoypypypy} shows. Moreover, the
fact that we have taken $T$ closed and $D(N)\subset D(T)$ is a
natural condition for it is tacitly assumed when $T\in B(H)$.
\end{rema}

It is almost surprising that the (most) generalized Fuglede-Putnam
version only needs $D(N)\subset D(T)$ (i.e. without requiring
$D(M)\subset D(T)$ or other conditions).

\begin{thm}\label{Main Theorem A closed  N M unbounded-2020}
Let $T$ be a closed operator with domain $D(T)$. Let $M$ and $N$ be
two unbounded normal operators with domains $D(N)$ and $D(M)$
respectively. If $D(N)\subset D(T)$, then
\[TN\subset MT\Longrightarrow TN^*\subset M^*T.\]
\end{thm}

\begin{proof}The proof is similar to that of Theorem
\ref{Fug-Put-MORTAD-PAMS-2003-2020}, so some details will be
omitted. Assume that $M$ and $N$ are two normal operators such that
 $D(N)\subset D(T)$ and $TN\subset MT$. Put
\[A=\left(
                            \begin{array}{cc}
                              0 & 0 \\
                              T & 0 \\
                            \end{array}
                          \right)\text{ and }\widetilde{N}=\left(
      \begin{array}{cc}
        N & 0 \\
        0 & M \\
      \end{array}
    \right).\]
Then
\[A\widetilde{N}=\left(
                            \begin{array}{cc}
                              0 & 0_{D(M)} \\
                              TN & 0 \\
                            \end{array}
                          \right)\text{ and }\widetilde{N}A=\left(
                            \begin{array}{cc}
                              0 & 0 \\
                              MT & 0 \\
                            \end{array}
                          \right).\]
Hence $A\widetilde{N}\subset \widetilde{N}A$. But
$D(\widetilde{N})=D(N)\oplus D(M)\subset D(T)\oplus H=D(A)$.
Therefore, the assumptions of Theorem
\ref{Fug-Put-MORTAD-PAMS-2003-2020} are fulfilled and so
$A\widetilde{N}^*\subset \widetilde{N}^*A$. Proceeding as in the end
of the proof of Theorem \ref{Fug-Put-MORTAD-PAMS-2003-2020} allows
us to finally obtain $TN^*\subset M^*T$, as wished.
\end{proof}


\begin{thebibliography}{1}


\bibitem{FUG}
B. Fuglede, \textit{ A Commutativity Theorem for Normal Operators,}
Proc. Nati. Acad. Sci., {\bf 36} (1950) 35-40.

\bibitem{MHM1}
M. H. Mortad, An Application of the Putnam-Fuglede Theorem to Normal
Products of Self-adjoint Operators, \textit{Proc. Amer. Math. Soc.},
{\bf 131/10}, (2003) 3135-3141.

\bibitem{Mortad-Fuglede-Putnham-All-unbd}
M. H. Mortad, An all-unbounded-operator version of the
Fuglede-Putnam theorem, \textit{Complex Anal. Oper. Theory},
\textbf{6/6} (2012) 1269-1273.

\bibitem{Paliogiannis-NEW-proofs-MORTAD}
F C. Paliogiannis,  A note on the Fuglede-Putnam theorem,
\textit{Proc. Indian Acad. Sci. Math. Sci.}, \textbf{123/2} (2013)
253-256.

\bibitem{paliogiannis-Fug-Putnam-ALL-UNBD}
F. C. Paliogiannis, A Generalization of the Fuglede-Putnam Theorem
to Unbounded Operators, \textit{J. Oper.}, 2015, Art. ID 804353, 3
pp.

\bibitem{PUT}
C. R. Putnam, \textit{ On Normal Operators in Hilbert Space,} Amer.
J. Math., {\bf 73}\textnormal{ (1951) 357-362}.

\bibitem{ROS.FP}
M. Rosenblum, \textit{On a Theorem of Fuglede and Putnam},  J. Lond.
Math. Soc., {\bf 33}, (1958) 376-377.

\bibitem{SCHMUDG-book-2012}
K. Schm\"{u}dgen. \textit{Unbounded Self-adjoint Operators on
Hilbert Space}, Springer. GTM {\bf 265}  (2012).

\bibitem{tretetr-book-BLOCK}
Ch. Tretter. \textit{Spectral Theory of Block Operator Matrices and
Applications}. Imperial College Press, London, 2008.










\end{thebibliography}
\end{document}